\documentclass[12pt]{article}

\usepackage{amssymb,amscd,amsthm, verbatim,amsmath,color,fancyhdr, mathrsfs,enumitem}
\usepackage{graphicx}
\usepackage{turnstile}
\usepackage{fullpage}
\usepackage{amssymb}
\usepackage{amsmath}

\usepackage[letterpaper, left=2.5cm, right=2.5cm, top=2.5cm,
bottom=2.5cm,dvips]{geometry}

\newtheorem{thm}{Theorem}[section]

\newtheorem{lem}[thm]{Lemma}
\newtheorem{rem}[thm]{Remark}

\newcommand{\N}{\mathbb{N}}
\makeatletter

\newlist{casenv}{enumerate}{4}
\setlist[casenv]{leftmargin=*,align=left,widest={iiii}}
\setlist[casenv,1]{label={{\itshape\ \casename} \arabic*.},ref=\arabic*}
\setlist[casenv,2]{label={{\itshape\ \casename} \roman*.},ref=\roman*}
\setlist[casenv,3]{label={{\itshape\ \casename\ \alph*.}},ref=\alph*}
\setlist[casenv,4]{label={{\itshape\ \casename} \arabic*.},ref=\arabic*}

\makeatother

\providecommand{\casename}{Case}

\date{\today}

\begin{document}
\title{\bf On the exponential Diophantine equation ${\displaystyle p\cdot 3^{x}+p^{y}=z^2}$ with $p$ a prime number}
\author{ Buosi, M.\footnote{Rodovia MGT 367 – Km 583, nº 5000, Alto da Jacuba, Diamantina-MG 39100-000, Brazil}; Ferreira, G.S. $^{\ast} $ and Porto, A.L.P.$^{\ast} $ \\
 marcelo.buosi@ufvjm.edu.br\\
 gilmar.ferreira@ufvjm.edu.br \\
 ander.porto@ufvjm.edu.br
 }
 
\maketitle

\begin{abstract}
In this paper we find non-negative integer solutions for exponential Diophantine equations of the type $p \cdot 3^x+ p^y=z^2,$ where $p$ is a prime number. We prove that such equation has a unique solution $\displaystyle{(x,y,z)=\left(\log_3(p-2), 0, p-1\right)}$ if $2 \neq p \equiv 2 \pmod 3$ and $(x,y,z)=(0,1,2)$ if $p=2$.  We also display the infinite solution set of that equation in the case $p=3$. Finally, a brief discussion of the case $p \equiv 1 \pmod 3$ is made, where we display an equation that does not have a non-negative integer solution and leave some open questions. 
The proofs are based on the use of the 
properties of the modular arithmetic.   
\end{abstract}

\noindent 2010 {\it Mathematics Subject Classification}: 11A07, 11A41, 11D61.\\
\noindent {\it Keywords}: Congruences; Exponential Diophantine equations.
\maketitle

\section*{Introduction}
Diophantine equations of the form $a^x + b^y = c^z$ have been studied by numerous mathematicians for many decades and by a variety of methods. One of the first references to these equations was given by Fermat-Euler \cite{euler}, showing that $(a, c)=(5,3)$ is the unique positive integer solution of the equation $a^2 + 2 = c^3$. Several works on exponential Diophantine equations have been developed in recent years. In 2011, Suvarnamani \cite{surva1} studied the Diophantine equation  $2^x + p^y = z^2$. Rabago \cite{rabago} studied the equations $3^x + 19^y = z^2$ and $3^x + 91^y = z^2.$ The solution sets are \{(1, 0, 2),(4, 1, 10)\} and \{(1, 0, 2),(2, 1, 10)\}, respectively. A. Suvarnamani {\it et al.} \cite{surva} found solutions of two Diophantine equations $4^x + 7^y = z^2$ and $4^x + 11^y = z^2$. Sroysang (see \cite{sro}) studied the Diophantine equation $3^x + 17^y=z^2$. Chotchaisthit (see \cite{chot}) showed that the Diophantine equation $p^x + (p+1)^y =z^2$ has unique solutions $(p,x,y,z)=(7,0,1,3)$  and $ (p,x,y ,z)=(3,2,2,5)$  if $(x,y,z) \in \N^3$ and $p$  is a Mersenne prime. In 2019, Thongnak {\it et al. } (see \cite{thong}) found exactly two non-trivial solutions for the equation $ 2^x - 3^y = z^2, $ namely $(1,0,1)$ and $(2,1,1).$ Buosi {\it et al. } (see \cite{buosi1} and \cite{buosi2}) studied some exponential Diophantine equations that generalized the work of Thongnak {\it et al. } (see \cite{thong}). 

 In this work we show that when $p>3$ is a prime integer such that $p \equiv 2 \pmod 3$, there is an ordered triple $(x,y,z)$ of non-negative integers that solves the equation $p \cdot 3^x + p^y=z^2$ if and only if $p-2$ is a non-trivial power of $3$. In the affirmative case, there exists only one solution which is given by \[(x,y,z)=(\log_3(p-2), 0, p-1).\] This result generalizes the theorem obtained in Thongnak {\it et al. } \cite{thong2} when $p=11.$ We also determine the unique solution of the case $p=2$ and the infinite set of solutions when $p = 3$.  The case where $p$ is congruent to $1$ modulo $3$ has not been solved completely because it is not understood why there are situations whose equation has a solution and others that do not. At the end of the article, a brief discussion of the case $p \equiv 1 \pmod 3$ is made, showing an example of an equation with no solution and suggesting some open questions.


\section*{Some notations}
Denote by $\mathbb{Z}$ be the set of integer numbers and let $\mathbb{N}$ be the set of all positive integers together with the number $0$, that is, $\mathbb{N}=\{0,1,2,3, \ldots\}$, such a set will be called the set of {\it natural numbers}. Define $\mathbb{N}^{\ast}= \mathbb{N} \backslash \{0\}$ and  $\mathbb{N}^{q} = \mathbb{N} \times \mathbb{N} \times \cdots \times  \mathbb{N}$ as the {\it cartesian product} of $q$ copies of $\mathbb{N}.$ When $a$ divides $b$ we will use the symbol $a\,|\,b.$  We will use the $\equiv$ symbol for congruence module $m$ and $a \equiv b\pmod m$ means that $a$ is congruent to $b$ module $m$. Let $a,m$ be integers with $a>0$ and $m >2.$  The smallest positive integer $k$ such that $a^k \equiv 1 \pmod m$ will be said the \emph{order} of $a$ modulo $m$ and will be denoted as $|a|_m =k$.  The set of all non-negative integer solutions of the equation $p \cdot 3^x + p^y = z^2$ will be said simply the \textit{solution set of the equation}, i. e., the set $\{(x,y,z) \in \mathbb{N}^{3}\,|\, p \cdot 3^x + p^y=z^2\}$.






\section*{Results}

In this section we will find the solution set for the equation \begin{equation}
    p \cdot 3^x+p^y=z^2, (x,y,z) \in \N^3,  
\end{equation} for several prime integers. We will divide the results into four sections: case $p=3$, general results for $p \neq 3$, case $p \equiv 2 \pmod 3$, and finally we will make a brief explanation of the case $p \equiv 1 \pmod 3$ since in this case the general problem still remains open. The motivation for this work is the paper Thongnak {\it et al. } \cite{thong2} where the authors solved the above equation in the particular case $p=11.$ The result of Thongnak {\it et al. } is an immediate consequence of Theorem \ref{generaliza thong}  proved in this article. 



\subsection*{Case $p=3$}\label{p=3}
 
In this subsection we present all the non-negative integer solutions of the equation $\displaystyle{p \cdot 3^x+p^y=z^2}$ in the particular case when $p=3$.

\begin{thm} \label{mainp3} The solution set of the Diophantine exponential equation \begin{equation}\label{eqdo3}
3 \cdot 3^x+3^y=z^2 
\end{equation}
in $\N^3$ is $\displaystyle{\left\{ \left(2n,2n,2 \cdot 3^n\right); \, n \in \N \right\} \cup \left\{ \left(1+2n,3+2n,2 \cdot 3^{n+1}\right); \, n \in \N \right\}.}$
\end{thm}

The proof is based on the combination of the results of the following six  lemmas.

\begin{lem}\label{0506231050} If $\displaystyle{(y,z) \in \N^2}$ is a solution of the equation  $\displaystyle{3^{y+1}+1=z^2}$ then $y=0$.
\end{lem}

\begin{proof} $3^{y+1}+1=z^2 \Longrightarrow 3^{y+1}=(z+1)(z-1) \Longrightarrow z-1=1, z+1=3 \Longrightarrow y=0$.
\end{proof}

\begin{lem}\label{0106231050} If $\displaystyle{(x,y,z) \in \N^3}$ is a solution of the equation  $\displaystyle{3^x\left(3^{y+1}+1\right)=z^2}$ then $y=0$.
\end{lem}

\begin{proof} Suppose there exists  $(x,y,z) \in \N^3$ such that  $\displaystyle{3^x\left(3^{y+1}+1\right)=z^2}$ and $y>0$.  By Lemma \ref{0506231050}, $3^{y+1}+1$ it is not a perfect square. Thus there is a prime integer $q$ that appears an odd number of times in the prime factorization of $3^{y+1}+1$. Since $q \, | \, z^2$ we have two possibilities:
\begin{align*}
\begin{cases}
    x=0  \Longrightarrow 3^{y+1}+1 \text{ is a perfect square; } \\
    x>0 \Longrightarrow q \, | \, 3^x \Longrightarrow  q = 3 \Longrightarrow 3 \, | \, 1. 
    \end{cases}
\end{align*}
In both cases we have an absurd. Therefore $y=0$. \end{proof}

\begin{lem}\label{0106231010} If $(x,y,z) \in \N^3$ is a solution of the equation  (\ref{eqdo3}) then $y -x \in \{0,1,2\}$.
\end{lem}
\begin{proof} If $y<x$ then there exists an integer $k> 0$ such that $x=y+k$. Replacing $x$ in (\ref{eqdo3}) with $y+k$  we obtain
\begin{align*}
3 \cdot 3^{y+k}+3^y=z^2 & \Longleftrightarrow 3^y\left(3^{k+1}+1\right)=z^2, 
\end{align*} 
which contradicts  Lemma \ref{0106231050}. Therefore $x \leq y.$ 

If $y-x \geqslant 3$ then $y=x+k$ for some integer $k \geqslant 3$. Replacing $y$ in \eqref{eqdo3} with $x+k$  we obtain 
\begin{align*}
 3 \cdot 3^x+3^{x+k}=z^2                     & \Longleftrightarrow 3^{x+1}\left(3^{k-1}+1\right)=z^2,
\end{align*} which is a contradiction with Lemma \ref{0106231050}. Therefore $y -x \in \{0,1,2\}$. \end{proof}

\begin{lem}\label{0106231233} If $(x,y,z) \in \N^3$ is a solution of the equation  (\ref{eqdo3}) then $x=y$ or $y=x+2$.
\end{lem}

\begin{proof} By Lema \ref{0106231010}, $y-x \in \{0,1,2\}$. Suppose $y=x+1$. Replacing $y$ in \eqref{eqdo3} with $x+1$ we obtain
\begin{align*}
    3^{x+1}+3^{x+1}=z^2 & \Longrightarrow 2 \cdot 3^{x+1} = z^2 \Longrightarrow 2 \, | \, z^2 \text{ and } 4 \text{ does not divide } z^2,
\end{align*} 
which is an absurd. Therefore $y -x \in \{0, 2\}$.
\end{proof}

\begin{lem}\label{0106231020} If $(x,y,z) \in \N^3$ is a solution of the equation  (\ref{eqdo3}) and $x=y$, then there exists  $n \in \N$ such that \[
x=y=2n \text{ and } z=2 \cdot 3^n.
\]
\end{lem}
\begin{proof}
Making $x=y$ in equation (\ref{eqdo3}) we get 
\begin{align*}
    3^{x+1}+3^x=z^2 & \Longrightarrow 4 \cdot 3^x = z^2 \Longrightarrow x \text{ is even}.
\end{align*}
Henceforth there exists $n \in \N$ such that $y=x=2n$  and $\displaystyle{z=\sqrt{4 \cdot 3^{2n}}=2 \cdot 3^n}$. 
\end{proof}

\begin{lem}\label{0106231200} If $(x,y,z) \in \N^3$ is a solution of the equation  (\ref{eqdo3}) and $y-x=2$, then there exists $n \in \N$ such that \[
x=1+2n, y=3+2n \text{ and } z=2 \cdot 3^{n+1}.
\]
\end{lem}

\begin{proof} Making $y=x+ 2$ in equation (\ref{eqdo3}) we get \begin{align*}
    3^{x+1}+3^{x+2}=z^2 & \Longrightarrow 4 \cdot 3^{x+1} = z^2 \Longrightarrow x \text{ is odd}.
\end{align*} 
Henceforth there exists $n \in \N$ such that $x=1+2n$, $y=3+ 2n$ and $\displaystyle{z=\sqrt{4 \cdot 3^{2n+2}}=2 \cdot 3^{n+1}}$.
\end{proof}

\subsection*{General results for a prime $p \neq 3$}

\begin{lem}\label{2105231820} Let $p \neq 3$ be a prime integer. If $(x,y,z) \in \mathbb{N}^3$ is a solution of \begin{equation}\label{main_equation}
    p \cdot 3^x+p^y=z^2, 
\end{equation} then $y=0$ or $y=1$.    
\end{lem}

\begin{proof} Let $ (x,y,z)$  be a solution of (\ref{main_equation}). Assume $y \geqslant 2$. It is clear that $z \neq 0.$  In this case, $p$ divides $z$ because
\[
p \cdot 3^x+p^y=z^2 \Longrightarrow p\left(3^x+p^{y-1}\right)=z^2.
\]
Let $m \in \N^*$ such that $z=mp$. Substitute $mp$ for $z$ in the above equation to obtain 
\[
3^x=p\left(m^2-p^{y-2}\right).
\]
If $x=0$ the above equation is an absurd for all prime integer $p \geq 2$. If $x>0$ we have $3 \, | \,p$ which is absurd for all prime integer $p \neq 3.$ Therefore $y=0$ or $y=1$. \end{proof}

If one substitute $0$ for $y$ in the equation \eqref{main_equation} one obtain $\displaystyle{p \cdot 3^x+1=z^2}$ which is equivalent to the following equation
\begin{align}
   p  \cdot 3^x & = z^2 -1= (z-1)(z+1). \label{second_equation}
\end{align}

\begin{lem}\label{2305231202} Let $p \neq 3$ be a prime integer. If $(x,z) \in \N^2$ is a solution of \eqref{second_equation}, then $x>0$ and $z \not\equiv 0 \pmod {3}$.    
\end{lem}

\begin{proof} Let $(x,z)$ be a solution of \eqref{second_equation}. If $x=0$ then 
\[ p=(z-1)(z+1) \Longrightarrow z=2 \text{ and } p=3,
\]
which is a contradiction. Hence $x>0$. If $3$ divides $z$ then $3$ divides $z^2-p\cdot 3^x =1$, which is an absurd. Therefore $z \not\equiv 0 \pmod {3}$. 
\end{proof}

We say that $h$ is a non-trivial power of $3$ if $h=3^x$ with $x \in \mathbb{N}^{*}$.

\begin{lem}\label{2305231206} Let $p>3$ be a prime integer. The equation \eqref{second_equation} has a solution in $\N^2$ if and only if $p-2$ is a non-trivial power of $3$ or $p+2$ is a non-trivial power  of $3$. In the affirmative case, the equation \eqref{second_equation} has a unique solution in $\N^2$  given by 
\begin{align*}
\left(\log_3(p-2),p-1\right) & \text{ if } p-2 \text{ is  a non-trivial power of } 3; \\
\left(\log_3(p+2),p+1\right) & \text{ if } p+2 \text{ is  a non-trivial power of } 3.
\end{align*}
\end{lem}

\begin{proof} Let $(x,z) \in \N^2$ a solution of \eqref{second_equation}. By Lemma \ref{2305231202}, $x>0$ and  $z \not\equiv 0 \pmod {3}$.

If $z \equiv 1 \pmod {3}$ then $z-1 \equiv 0 \pmod {3}$. Since $\displaystyle{ p \cdot 3^x=(z-1)(z+1)}$ it follows that 
\[
\begin{cases}
    z+1=p \\
    z-1 = 3^x
\end{cases}
\Longrightarrow 
\begin{cases}
z=p-1 \\
3^x=z-1=p-2
\end{cases}
\Longrightarrow
\begin{cases}
z=p-1 \\
x=\log_3(p-2).
\end{cases}
\]

If $z \equiv 2 \pmod {3}$ then $z+1 \equiv 0 \pmod {3}$. Since $\displaystyle{ p \cdot 3^x=(z-1)(z+1)}$ it follows that 
\[
\begin{cases}
    z-1=p \\
    z+1 = 3^x
\end{cases}
\Longrightarrow 
\begin{cases}
z=p+1 \\
3^x=z+1=p+2
\end{cases}
\Longrightarrow
\begin{cases}
z=p+1 \\
x=\log_3(p+2).
\end{cases}
\]
The converse is straightforward and will be omitted. \end{proof}

Making $y=1$ in the equation \eqref{main_equation} one obtain
\begin{equation}\label{third_equation}
 p \cdot 3^x+p=z^2.   
\end{equation}

\begin{lem}\label{2405230800} Let $p>3$ be a prime integer. If $(x,z) \in \N^2$ is a solution of \eqref{third_equation} then $x>0$ and $z \not\equiv 0 \pmod {3}$.    
\end{lem}

\begin{proof} Suppose $(x,z) \in \N^2$  is a solution of \eqref{third_equation}.

If $x=0$ then $2p=z^2$, which is an absurd. It follows that $x>0$.

If $z \equiv 0 \pmod {3}$ then $3$ divides $z^2=p \cdot 3^x+p$ and therefore $3$ divides $p$ which is a contradiction.
\end{proof}

\subsection*{Case $p \equiv 2 \pmod 3$}

The Theorem \ref{generaliza thong} below presents all the non-negative integer solutions of the equation $\displaystyle{p \cdot 3^x+p^y=z^2}$ in the particular case where  $p \equiv 2 \pmod 3$ and $p \neq 2$. This result generalizes Theorem 2.1 of \cite{thong2} where $p=11$.

\begin{lem} \label{z2naoe2} There is no $z \in \mathbb{Z}$ such that $z^2 \equiv 2 \pmod 3$.
\end{lem}

\begin{proof} If $z \equiv 0 \pmod 3$ then $z^2 \equiv 0 \pmod 3$. If $z \equiv 1 \pmod 3$ or $z \equiv 2 \pmod 3$ then $z^2 \equiv 1 \pmod 3$.
\end{proof}

\begin{thm}\label{generaliza thong} Let $p > 3$ be a prime integer such that $p \equiv 2 \pmod 3.$  The equation 
\begin{equation}\label{second_main_eq}
p\cdot 3^{x}+p^{y}=z^{2},
\end{equation} 
admits a solution in $\N^3$ if and only if $p-2$ is a non-trivial power of $3$. In the affirmative case, the unique solution is $(x,y,z)=(\log_3(p-2), 0, p-1)$. 
\end{thm}

\begin{proof} Let $(x,y,z)$ be a solution of (\ref{second_main_eq}).  
By Lemma \ref{2105231820} we must have $y=0$ or $y=1.$ If $y=0,$ it follows from Lemma \ref{2305231206}  that $(\log_3(p-2), 0, p-1)$ is the unique solution in $\N^3$ of the equation $p\cdot 3^{x}+1=z^{2},$ since $\log_3(p-2)$ is an integer. Now consider $y=1.$ By Lemma \ref{2405230800}, $x \geq 1$.   
So we get \[2 \equiv p \equiv z^2 - p\cdot  3^x \equiv z^2 \pmod 3\] which is a contradiction by Lemma \ref{z2naoe2}.
\end{proof}

\begin{rem}
For example, for $p=17, 23, 41, 53, 59, 71$ the equation of the previous theorem has no non-negative integer solutions. For $p=5, 11, 29, 83$ the solutions are respectively $(1,0,4)$, $(2,0,10)$, $(3,0,28)$ and $(4,0,82)$.  
\end{rem}

\begin{thm} The unique solution of the exponential Diophantine equation  \begin{equation}\label{eqteorema3.6}
2\cdot 3^{x}+2^{y}=z^{2}, \,\, (x,y,z) \in \N^3,  
\end{equation} is the ordered triple $(x,y,z)=(0, 1, 2).$  
\end{thm}

\begin{proof}
Let $(x,y,z)$ be a solution of (\ref{eqteorema3.6}). By Lemma \ref{2105231820} we must have $y=0$ or $y=1.$ If $y=0,$   it follows from Lemma \ref{2305231202} that $x>0$ and $z \not\equiv 0 \pmod {3}$. In this case we have the following equivalence for equation (\ref{eqteorema3.6}) \[2\cdot 3^{x}+1=z^{2} \Longleftrightarrow  2\cdot  3^x= (z-1) \cdot  (z+1) = z^2 - 1.\] If $z \equiv 1 \pmod {3}$ then $z-1 \equiv 0 \pmod {3}$ and $z+1 \equiv 2 \pmod {3}$ then we have 
\[
\begin{cases}
    z+1=2 \\
    z-1 = 3^x
\end{cases}
\Longrightarrow 
\begin{cases}
z=1 \\
3^x=0,
\end{cases}
\]
an absurd. 

If $z \equiv 2 \pmod {3}$ then $z-1 \equiv 1 \pmod {3}$ and $z+1 \equiv 0 \pmod {3}$ then we have 
\[
\begin{cases}
    z+1=3^x \\
    z-1 = 2
\end{cases}
\Longrightarrow 
\begin{cases}
z=3 \\
3^x=4,
\end{cases}
\]
an absurd. 

Now consider $y=1.$ In this case equation (\ref{eqteorema3.6}) reduces to equation  \[2\cdot 3^{x}+2=2\cdot (3^x +1)=z^{2}.\] If $x=0$ we have $z^2=4,$ so $z=2.$  Therefore $(x,y,z)=(0, 1, 2)$ is a solution to equation (\ref{eqteorema3.6}) in $\N^3.$ If $x>0$ then $z^2 \equiv 2 \cdot (3^x +1) \equiv 2  \pmod {3}$, a contradiction by Lemma \ref{z2naoe2}. Therefore $(x,y,z)=(0, 1, 2)$ is the only solution of equation (\ref{eqteorema3.6}). 
\end{proof}

\subsection*{A brief discussion of the case $p \equiv 1 \pmod 3$}

When $p \equiv 1 \pmod 3,$  the equation \begin{equation} \label{open p 1 mod 3}
    p \cdot 3^x+p^y=z^2, (x,y,z) \in \N^3,  
\end{equation} has not yet been completely solved, that is, the behavior of the solutions of these equations is not known, whether they have a solution and whether the solutions, if any, are finite or infinite. 

Let $(x,y,z)$ be a solution of (\ref{open p 1 mod 3}). By Lemma \ref{2105231820}, $y \in \{0, 1\}$. By Lemma \ref{2305231206} we can say whether equation (\ref{open p 1 mod 3}) will have a solution as long as $p+2$ is a non-trivial power of $3.$  Furthermore, that lemma determines the unique solution in this case.  However, for the case $y=1$ we do not have a conclusive result for the time being. For example, equations with $p=7, 61$ and $547$ respectively have the following solutions $(3,1,14), (5,1,122)$ and $(7,1, 1094)$. We do not know if those three equations have other solutions. 


\begin{rem}
Note that $(q,2) \in \N^2$ is a solution of $3^x+1=p \cdot w^2$ if and only if $(q, 1, 2p)$ is a solution of $p \cdot 3^x+p^y=z^2, (x,y,z) \in \N^3$.    \end{rem} 

In the next theorem we will show an example whose given equation does not have non-negative integer solutions.


	\begin{thm}
		The exponential Diophantine equation \begin{equation}\label{eq p=13}
			13 \cdot 3^x + 13^y = z^2, \, (x,y,z) \in \N^3,
		\end{equation}
	has no solutions.
	\end{thm}

	\begin{proof}
Let $ (x,y,z)$  be a solution of (\ref{eq p=13}). By Lemma \ref{2105231820}, $y \in \{0, 1\}$. First consider $y=0.$ By Lemma \ref{2305231206} there are no solutions to the equation in this case, since $p+2=15$ is not a non-trivial power of $3.$ 

Suppose there is a solution $(x,1,z) \in \N^3$ of (\ref{eq p=13}). In this case equation (\ref{eq p=13}) reduces to $13 \cdot 3^x + 13 = z^2.$   Note that $13$ divides $z$ and therefore $z=13 \cdot w, \, w \in \N^{*}$. So we have the following equivalence of equations \begin{equation}\label{eq 13 equivalentes}
13 \cdot 3^x + 13 =z^{2}=13^{2} \cdot w^2  \Longleftrightarrow 3^{x} +1 = 13 \cdot w^2 \equiv 0 \pmod {13}. 
\end{equation}
On the other hand, notice that $3^2 \equiv 9 \pmod { 13}$ and $3^3 \equiv 1 \pmod  {13}$. Therefore the order of $3$ modulo $13$ is equal to $3,$  that is $ |3|_{13}=3.$ So write $x = 3m + r$, where $m \in \N$ and $r \in \{0,1,2\}$.  So we have the following equation \begin{equation*}
			3^x +1 = 3^{3m + r} +1 = 27^m \cdot 3^r +1 \equiv 3^r +1\pmod {13} \equiv \left\{
			\begin{array}{ccc}
				2 \pmod {13} & \mbox{if } r = 0 \\
				4 \pmod {13} & \mbox{if } r = 1 \\
				10 \pmod {13} & \mbox{if } r=2,
			\end{array}
			\right.
		\end{equation*} an absurd. 
		\end{proof}

\section*{Open questions} 

The following questions refer to the equation \begin{equation} \label{open}
    p \cdot 3^x+p^y=z^2, (x,y,z) \in \N^3 \mbox{ with  } p \equiv 1 \pmod 3.  
\end{equation}

\begin{itemize}
    \item When $p \equiv 1 \pmod 3,$  what additional conditions must exist on $p$ for the equation (\ref{open}) to have a solution?  
    
    

\item If there is a solution for equation (\ref{open}), how do you know if the number of solutions is finite or infinite?

\item What additional conditions must be imposed on $p$ for there to be a unique solution?


\end{itemize}


\end{document}